\documentclass[11pt]{article}

\usepackage[a4paper,margin=1in]{geometry}
\usepackage[T1]{fontenc}
\usepackage[utf8]{inputenc}
\usepackage{lmodern}
\usepackage{amsmath,amssymb,amsthm,mathtools}
\usepackage{enumitem}
\usepackage{hyperref}
\usepackage{booktabs}
\usepackage{array}
\usepackage{graphicx}
\usepackage{float}
\usepackage{microtype}

\hypersetup{
  colorlinks=true,
  linkcolor=blue,
  citecolor=blue,
  urlcolor=blue
}

\newtheorem{theorem}{Theorem}[section]
\newtheorem{proposition}[theorem]{Proposition}
\newtheorem{lemma}[theorem]{Lemma}
\newtheorem{corollary}[theorem]{Corollary}
\newtheorem{problem}[theorem]{Problem}
\newtheorem{observation}[theorem]{Observation}
\theoremstyle{definition}
\newtheorem{definition}[theorem]{Definition}
\theoremstyle{remark}
\newtheorem{remark}[theorem]{Remark}

\newcommand{\Cl}{\mathrm{Cl}}
\newcommand{\supp}{\operatorname{supp}}

\title{Degree theory of the partition graph: exact maxima, profiles, and fibres}
\author{Fedor B. Lyudogovskiy}
\date{}

\begin{document}

\maketitle

\begin{abstract}
We develop a systematic degree theory for the partition graph $G_n$, whose vertices are the partitions of $n$ and whose edges correspond to elementary unit transfers between parts. The starting point is the exact local degree formula, expressed in terms of support size, gap pattern, and multiplicity pattern, which we rewrite in a support-plus-bonus form.

Writing
\[
n=T_s+q,
\qquad
T_s=\frac{s(s+1)}2,
\qquad
0\le q\le s,
\]
we prove that every degree-maximizing partition of $n$ lies in the support-maximal stratum, and we obtain the exact formula
\[
\Delta_n
=
s(s-1)+\left\lfloor\sqrt{4q+1}\right\rfloor-1
\]
for the maximal degree in $G_n$. We then move from value theory to profile theory. For a support-maximal partition, let $A(\lambda)$ and $B(\lambda)$ denote the numbers of active gap bonuses and multiplicity bonuses, respectively. We show that the set of realized maximizing profiles is
\[
\Pi_n
=
\bigl\{(a,b)\in\mathbb Z_{\ge0}^2 : a+b=\rho(q),\ T_a+T_b\le q\bigr\},
\]
where $\rho(q)=\left\lfloor\sqrt{4q+1}\right\rfloor-1$.

This reduces the remaining extremal degree theory to the fibre-level geometry of the maximizer set. For each realized profile $(a,b)$, we study the corresponding fibre of maximizers. We prove that every such fibre is nonempty, construct canonical representatives, obtain lower bounds for mixed fibres, and show that conjugation induces a bijection between the fibres for $(a,b)$ and $(b,a)$.

In this way, the degree theory of $G_n$ splits naturally into three levels: exact value theory, exact profile theory, and fibre-level geometry. The first few near-triangular windows are classified completely.
\end{abstract}



\section{Introduction}

For a positive integer $n$, let $G_n$ denote the partition graph of $n$: its vertices are the integer partitions $\lambda\vdash n$, and two vertices are adjacent if one partition is obtained from the other by an elementary transfer of one unit from one part to another, followed by reordering. Standard background on integer partitions, Ferrers diagrams, and conjugation may be found in Andrews~\cite{Andrews} and Stanley~\cite{StanleyEC2}. The elementary transfer operation used in the definition of $G_n$ is also closely related to the transfer operations that occur in majorization theory; see, for example, Marshall, Olkin, and Arnold~\cite{MarshallOlkinArnold}. In earlier papers of this cycle, the graph $G_n$ was studied from several complementary viewpoints, including local morphology~\cite{LyuLocal}, axial and boundary structure~\cite{LyuAxial,LyuBoundary}, directional geometry~\cite{LyuDirectional}, growth across $n$~\cite{LyuGrowing}, and the topology of the clique complex $K_n=\Cl(G_n)$~\cite{LyuClique}. The purpose of the present paper is to isolate and develop one specific invariant: the vertex degree in $G_n$.

The degree has an exact expression in terms of a partition's support size, gap pattern, and multiplicity pattern. We show that this local formula has strong global consequences. It leads to an extremal theory of degrees, an exact description of the maximal degree, and a natural structural decomposition of the set of degree-maximizing partitions.

The basic mechanism is as follows. For a partition with support size $r$, the degree decomposes as
\[
\deg(\lambda)=r(r-1)+A(\lambda)+B(\lambda),
\]
with a support-dependent baseline term $r(r-1)$ and a binary bonus term determined by which gaps exceed $1$ and which multiplicities exceed $1$. This separates the degree theory into two layers: a coarse layer governed by support size, and a finer layer governed by binary local data. Once this decomposition is combined with the staircase decomposition
\[
n=T_s+q,
\qquad
T_s=\frac{s(s+1)}2,
\qquad
0\le q\le s,
\]
the global extremal problem becomes tractable.

Our first main result is an exact formula for the maximal degree. We prove that every degree-maximizing partition of $n$ lies in the support-maximal stratum, and that
\[
\Delta_n
=
s(s-1)+\left\lfloor\sqrt{4q+1}\right\rfloor-1
\qquad
(n=T_s+q,\ 0\le q\le s).
\]
Thus the maximal degree is completely explicit. In particular, the triangular numbers $T_s$ form a rigid extremal backbone, and the increment above this baseline is governed by a step function of the excess parameter $q$.

Our second main result is profile-theoretic. For a support-maximal partition $\lambda$, let $A(\lambda)$ denote the number of active gap bonuses and $B(\lambda)$ the number of active multiplicity bonuses. We show that every maximizer satisfies
\[
A(\lambda)+B(\lambda)=\left\lfloor\sqrt{4q+1}\right\rfloor-1,
\]
together with an explicit triangular budget constraint on the pair $(A(\lambda),B(\lambda))$. Moreover, this constraint is exact: every admissible maximizing profile is realized by an actual maximizer. Hence the profile-level structure of the maximizer set is determined completely. Thus the exact global theory stops at the profile level; the remaining complexity is fibre-level.

Once $\Delta_n$ and the set of realized maximizing profiles are known, the remaining degree theory reduces to the fibre level. For each realized profile $(a,b)$, one studies the fibre
\[
M_n^{(a,b)}
=
\{\lambda\vdash n : \deg(\lambda)=\Delta_n,\ (A(\lambda),B(\lambda))=(a,b)\}.
\]
At this level, the main questions concern the number of maximizers, the sizes of the individual fibres, conjugation symmetry, the self-conjugate part of the maximizer set, and the recurrence of geometric pattern types inside the support-maximal layer. We prove several general structural results in this direction, including the existence of canonical representatives for every realized profile, explicit lower bounds for mixed fibres, and a precise description of the action of conjugation on profile coordinates. We also determine exactly the first near-triangular fibre windows, showing that the unresolved part of the fibre geometry begins only after these initial excesses.

This leads naturally to a three-level organization of the paper: value theory, profile theory, and fibre-level structure.

The main contributions of the paper are as follows.

\begin{enumerate}[leftmargin=2em]
\item
We rewrite the exact local degree formula in a form adapted to global optimization, separating a support baseline from a binary bonus term.

\item
We prove that every global degree maximizer lies in the support-maximal stratum.

\item
We obtain an exact closed formula for the maximal degree:
\[
\Delta_n
=
s(s-1)+\left\lfloor\sqrt{4q+1}\right\rfloor-1
\qquad
(n=T_s+q).
\]

\item
We determine the realized maximizing profile set exactly.

\item
We establish several general fibre-level structural results, including canonical representatives, conjugation symmetry, and explicit lower bounds for mixed fibres.

\item
We classify exactly the first near-triangular fibre windows and isolate the localization mechanism that governs the unresolved fixed-$q$ regime.
\end{enumerate}

This paper is a detailed special study within the broader partition-graph project. Its aim is to develop the degree not merely as a local invariant, but as an extremal theory and a source of structural and numerical questions.

The paper is organized as follows. Section~\ref{sec:preliminaries} recalls the degree formula and establishes the support-maximal reduction. Section~\ref{sec:extremal} develops the exact extremal theory and proves the formula for $\Delta_n$. Section~\ref{sec:small-excess} examines the first near-triangular layers as an initial rigid regime. Section~\ref{sec:profiles} determines all realized maximizing profiles. Section~\ref{sec:fibres} develops the first general fibre-level structure theory. Section~\ref{sec:computations} classifies the initial near-triangular fibre windows, proves a localization lemma for active coordinates, and formulates the resulting stability questions. The final section collects conclusions and open problems.

\section{Preliminaries and the degree formula}\label{sec:preliminaries}

Recall that $G_n$ is the graph whose vertices are the partitions of $n$, with adjacency given by elementary unit transfers between parts followed by reordering. We use standard partition and Ferrers-diagram terminology in the sense of~\cite{Andrews,StanleyEC2}. Conjugation of Ferrers diagrams preserves this adjacency, and therefore defines an involutive automorphism of $G_n$.

We begin by introducing the degree-theoretic language used throughout the paper. Our main technical input is the previously established exact degree formula in terms of the ordered local transfer data of a partition. Here we use that formula globally, as a structural tool for studying extremal degree behavior, maximizing profiles, and the organization of high-degree regions in the partition graph.

We write a partition $\lambda\vdash n$ in compressed form as
\[
\lambda=(L_1^{m_1},\dots,L_r^{m_r}),
\qquad
L_1>\cdots>L_r>0,
\]
where $r=\supp(\lambda)$ is the number of distinct part sizes. We define
\[
g_i=L_i-L_{i+1}\quad(1\le i<r),
\qquad
g_r=L_r,
\]
and
\[
\alpha_i=\mathbf 1_{\{g_i>1\}},
\qquad
\beta_i=\mathbf 1_{\{m_i>1\}}.
\]
Further, let
\[
A(\lambda):=\sum_{i=1}^r \alpha_i,
\qquad
B(\lambda):=\sum_{i=1}^r \beta_i.
\]
For $m\ge 0$, write
\[
T_m:=\frac{m(m+1)}2,
\]
so in particular $T_0=0$.

\begin{lemma}[Exact excess identity]\label{lem:excess-identity}
Let
\[
\lambda=(L_1^{m_1},\dots,L_r^{m_r})\vdash n
\]
be written in compressed form. Then
\[
n-T_r
=
\sum_{i=1}^r i\,(g_i-1)
+
\sum_{i=1}^r (m_i-1)L_i.
\]
Consequently,
\[
A(\lambda)+B(\lambda)\le n-T_r.
\]
\end{lemma}

\begin{proof}
Since
\[
\lambda=(L_1^{m_1},\dots,L_r^{m_r}),
\]
we have
\[
n=\sum_{i=1}^r m_iL_i
 =\sum_{i=1}^r L_i+\sum_{i=1}^r (m_i-1)L_i.
\]
By the definition of the gaps,
\[
L_i=\sum_{j=i}^r g_j.
\]
Therefore
\[
\sum_{i=1}^r L_i
=
\sum_{i=1}^r\sum_{j=i}^r g_j
=
\sum_{j=1}^r j\,g_j.
\]
Hence
\[
n=\sum_{j=1}^r j\,g_j+\sum_{i=1}^r (m_i-1)L_i.
\]
Subtracting
\[
T_r=\sum_{j=1}^r j
\]
gives
\[
n-T_r
=
\sum_{j=1}^r j(g_j-1)+\sum_{i=1}^r (m_i-1)L_i.
\]

If $g_i>1$, then $g_i-1\ge 1$, so $\alpha_i\le i(g_i-1)$. If $m_i>1$, then $(m_i-1)L_i\ge 1$, so $\beta_i\le (m_i-1)L_i$. Summing yields
\[
A(\lambda)\le \sum_{i=1}^r i(g_i-1),
\qquad
B(\lambda)\le \sum_{i=1}^r (m_i-1)L_i,
\]
and hence
\[
A(\lambda)+B(\lambda)\le n-T_r.
\]
\end{proof}

\begin{theorem}[Fixed-support degree formula and equality cases]\label{thm:fixed-support-degree}
Let
\[
\lambda=(L_1^{m_1},\dots,L_r^{m_r})\vdash n
\]
be written in compressed form. Then
\[
\deg(\lambda)=r(r-1)+A(\lambda)+B(\lambda).
\]
In particular,
\[
r(r-1)\le \deg(\lambda)\le r(r+1).
\]

Moreover:
\begin{enumerate}[leftmargin=2em]
\item
\[
\deg(\lambda)=r(r-1)
\]
if and only if
\[
\lambda=\delta_r:=(r,r-1,\dots,1).
\]

\item
\[
\deg(\lambda)=r(r+1)
\]
if and only if
\[
g_i>1\ \text{for all }i
\qquad\text{and}\qquad
m_i>1\ \text{for all }i.
\]
\end{enumerate}
\end{theorem}

\begin{proof}
By the exact degree formula established in~\cite{LyuLocal},
\[
\deg(\lambda)=r(r-1)+A(\lambda)+B(\lambda).
\]
Thus, for fixed support size, the degree is controlled entirely by the binary indicators of large gaps and repeated part sizes.

Since each $\alpha_i,\beta_i\in\{0,1\}$, we have
\[
0\le A(\lambda)\le r,
\qquad
0\le B(\lambda)\le r,
\]
and therefore
\[
r(r-1)\le \deg(\lambda)\le r(r+1).
\]

For the lower equality, we have
\[
\deg(\lambda)=r(r-1)
\iff
A(\lambda)=B(\lambda)=0.
\]
Thus $g_i=1$ and $m_i=1$ for all $i$, so
\[
L_i=\sum_{j=i}^r g_j=r-i+1,
\]
and hence
\[
\lambda=(r,r-1,\dots,1)=\delta_r.
\]
Conversely, $\delta_r$ satisfies $A(\delta_r)=B(\delta_r)=0$.

For the upper equality, we have
\[
\deg(\lambda)=r(r+1)
\iff
A(\lambda)=r \ \text{and}\ B(\lambda)=r,
\]
which is equivalent to
\[
g_i>1\ \text{for all }i
\qquad\text{and}\qquad
m_i>1\ \text{for all }i.
\]
\end{proof}

\begin{proposition}[Minimum degree]\label{prop:min-degree}
For every $n\ge 2$,
\[
\min_{\lambda\vdash n}\deg(\lambda)=1.
\]
This minimum is attained exactly at the two antenna vertices
\[
(n)
\qquad\text{and}\qquad
(1^n).
\]
For $n=1$, the unique vertex has degree $0$.
\end{proposition}

\begin{proof}
The partition $(n)$ has compressed form $(n^1)$, so $r=1$, $g_1=n$, and $m_1=1$. Hence
\[
\deg((n))=1.
\]
By conjugation symmetry, the same holds for $(1^n)$.

Now let $\lambda\vdash n$ be any other partition with $n\ge 2$. If $\supp(\lambda)\ge 2$, then
\[
\deg(\lambda)\ge r(r-1)\ge 2.
\]
So any degree-$1$ partition must have support size $1$, hence must be of the form $(a^m)$ with $am=n$. If $m=1$, then $\lambda=(n)$. If $a=1$, then $\lambda=(1^n)$. Otherwise $a>1$ and $m>1$, so $A(\lambda)=B(\lambda)=1$, hence $\deg(\lambda)=2$.
\end{proof}

\section{Exact extremal degree theory}\label{sec:extremal}

Let
\[
s=s(n):=\max\Bigl\{r:\frac{r(r+1)}2\le n\Bigr\}.
\]
Equivalently,
\[
T_s\le n<T_{s+1},
\qquad
T_s=\frac{s(s+1)}2.
\]
We also write
\[
n=T_s+q,
\qquad
0\le q\le s.
\]

\begin{proposition}[Triangular numbers]\label{prop:triangular}
Let
\[
n=T_s=\frac{s(s+1)}2.
\]
Then
\[
\Delta_n=s(s-1),
\]
and the unique degree-maximizing partition of $n$ is the staircase partition
\[
\delta_s=(s,s-1,\dots,1).
\]
\end{proposition}

\begin{proof}
Let $\lambda\vdash T_s$ and write $r=\supp(\lambda)$. Since $T_r\le T_s$, we have $r\le s$.

If $r\le s-2$, then
\[
\deg(\lambda)\le r(r+1)\le (s-2)(s-1)<s(s-1).
\]
If $r=s-1$, then Theorem~\ref{thm:fixed-support-degree} gives
\[
\deg(\lambda)\le r(r+1)=s(s-1).
\]
If equality held, then we would have
\[
A(\lambda)=B(\lambda)=r=s-1.
\]
For $s\ge 3$, this would imply
\[
A(\lambda)+B(\lambda)=2(s-1)>s=T_s-T_{s-1},
\]
contradicting Lemma~\ref{lem:excess-identity}. For $s=2$, we have $r=1$, so $\lambda$ has
support size $1$; but by Proposition~\ref{prop:min-degree} every such partition has degree $1$,
not $2=s(s-1)$. Hence in all cases
\[
\deg(\lambda)<s(s-1)
\qquad (r\le s-1).
\]
If $r=s$, then Lemma~\ref{lem:excess-identity} gives
\[
A(\lambda)+B(\lambda)\le T_s-T_s=0,
\]
hence $A(\lambda)=B(\lambda)=0$, so $\lambda=\delta_s$ and $\deg(\lambda)=s(s-1)$.
Therefore $\delta_s$ is the unique degree-maximizing partition.
\end{proof}

\begin{theorem}[Support-maximality of global degree maximizers]\label{thm:support-maximal}
Let
\[
n=T_s+q,
\qquad
1\le q\le s.
\]
If $\lambda\vdash n$ satisfies
\[
\deg(\lambda)=\Delta_n,
\]
then
\[
\supp(\lambda)=s.
\]
\end{theorem}

\begin{proof}
Since
\[
T_s<n<T_{s+1},
\]
every partition of $n$ has support size at most $s$.

Let $\mu\vdash n$ have support size $r\le s-1$. Then
\[
\deg(\mu)\le r(r+1)\le (s-1)s=s(s-1).
\]

Now consider
\[
\lambda_q=(s+q,s-1,s-2,\dots,1)\vdash n.
\]
It has support size $s$, all multiplicities equal to $1$, and exactly one active gap bonus, so
\[
\deg(\lambda_q)=s(s-1)+1.
\]
Hence no maximizer can have support at most $s-1$.
\end{proof}

\begin{corollary}[Reduction to the support-maximal layer]\label{cor:support-maximal-reduction}
Let
\[
n=T_s+q,
\qquad
0\le q\le s.
\]
Then
\[
\Delta_n
=
s(s-1)+
\max\bigl\{A(\lambda)+B(\lambda): \lambda\vdash n,\ \supp(\lambda)=s\bigr\}.
\]
In particular,
\[
s(s-1)\le \Delta_n\le s(s-1)+q.
\]
Moreover, if $q>0$, then
\[
\Delta_n\ge s(s-1)+1.
\]
\end{corollary}

\begin{proof}
If $q=0$, the claim follows from Proposition~\ref{prop:triangular}. If $q>0$, apply Theorem~\ref{thm:support-maximal} and then Theorem~\ref{thm:fixed-support-degree}. The upper bound follows from Lemma~\ref{lem:excess-identity}.
\end{proof}

\begin{lemma}[Separate triangular lower bounds for gap and multiplicity bonuses]\label{lem:separate-triangular-bounds}
Let
\[
n=T_s+q,
\qquad
0\le q\le s,
\]
and let $\lambda\vdash n$ have support size $s$. Write
\[
A(\lambda)=a,
\qquad
B(\lambda)=b.
\]
Then
\[
q\ge T_a+T_b.
\]
\end{lemma}

\begin{proof}
By Lemma~\ref{lem:excess-identity},
\[
q
=
\sum_{i=1}^s i(g_i-1)+\sum_{i=1}^s (m_i-1)L_i.
\]
If $I=\{i:g_i>1\}$ and $J=\{i:m_i>1\}$, then $|I|=a$ and $|J|=b$, so
\[
q\ge \sum_{i\in I} i+\sum_{j\in J} L_j.
\]
The first sum is at least $T_a$, and the second is at least $T_b$, since the relevant part sizes are distinct positive integers.
\end{proof}

\begin{definition}\label{def:rho}
For $q\ge 0$, define
\[
\rho(q):=\max\{a+b:\ a,b\ge 0,\ T_a+T_b\le q\}.
\]
\end{definition}

\begin{proposition}[Explicit formula for $\rho(q)$]\label{prop:rho-explicit}
For every $q\ge 0$,
\[
\rho(q)=\left\lfloor\sqrt{4q+1}\right\rfloor-1.
\]
Equivalently, $\rho(q)$ is the largest integer $c\ge 0$ such that
\[
\left\lfloor\frac{(c+1)^2}{4}\right\rfloor\le q.
\]
\end{proposition}

\begin{proof}
Fix $c\ge 0$. Among pairs $(a,b)$ with $a+b=c$, the quantity $T_a+T_b$ is minimized when $a$ and $b$ are as close as possible. Writing $c=2m$ or $c=2m+1$, we obtain
\[
\min_{a+b=c}(T_a+T_b)
=
\begin{cases}
m(m+1), & c=2m,\\
(m+1)^2, & c=2m+1.
\end{cases}
\]
This equals
\[
\left\lfloor\frac{(c+1)^2}{4}\right\rfloor.
\]
Thus $\rho(q)$ is the largest $c$ such that
\[
\left\lfloor\frac{(c+1)^2}{4}\right\rfloor\le q.
\]

Set
\[
m:=\left\lfloor\sqrt{4q+1}\right\rfloor.
\]
Then
\[
m^2\le 4q+1 < (m+1)^2.
\]
For $c=m-1$, we have
\[
\left\lfloor\frac{(c+1)^2}{4}\right\rfloor
=
\left\lfloor\frac{m^2}{4}\right\rfloor
\le
\left\lfloor\frac{4q+1}{4}\right\rfloor
=q,
\]
so $c=m-1$ is admissible.

For $c=m$, we have
\[
(c+1)^2=(m+1)^2>4q+1.
\]
Since $(m+1)^2\equiv 0$ or $1\pmod 4$, this strict inequality implies
\[
(m+1)^2\ge 4q+4.
\]
Therefore
\[
\left\lfloor\frac{(c+1)^2}{4}\right\rfloor
=
\left\lfloor\frac{(m+1)^2}{4}\right\rfloor
\ge q+1,
\]
so $c=m$ is not admissible.

Hence the largest admissible $c$ is $m-1$, that is,
\[
\rho(q)=m-1=\left\lfloor\sqrt{4q+1}\right\rfloor-1.
\]
\end{proof}

\begin{theorem}[Exact maximal degree]\label{thm:exact-max-degree}
Let
\[
n=T_s+q,
\qquad
0\le q\le s.
\]
Then
\[
\Delta_n=s(s-1)+\rho(q).
\]
Equivalently,
\[
\Delta_n
=
s(s-1)+\left\lfloor\sqrt{4q+1}\right\rfloor-1.
\]
\end{theorem}

\begin{proof}
By Corollary~\ref{cor:support-maximal-reduction},
\[
\Delta_n
=
s(s-1)+
\max\bigl\{A(\lambda)+B(\lambda):\lambda\vdash n,\ \supp(\lambda)=s\bigr\}.
\]
If $\lambda\vdash n$ has support size $s$ and
\[
A(\lambda)=a,
\qquad
B(\lambda)=b,
\]
then Lemma~\ref{lem:separate-triangular-bounds} gives
\[
T_a+T_b\le q.
\]
By Definition~\ref{def:rho}, this implies
\[
A(\lambda)+B(\lambda)\le \rho(q).
\]
Therefore
\[
\Delta_n\le s(s-1)+\rho(q).
\]

It remains to construct a partition $\lambda\vdash n$ with support size $s$ such that
\[
A(\lambda)+B(\lambda)=\rho(q).
\]

Set
\[
r:=\rho(q),
\qquad
a:=\left\lfloor\frac r2\right\rfloor,
\qquad
b:=\left\lceil\frac r2\right\rceil.
\]
By Proposition~\ref{prop:rho-explicit},
\[
T_a+T_b\le q.
\]
Let
\[
\ell:=q-T_a-T_b\ge 0.
\]
Since $a+b=r=\rho(q)\le q\le s$, we have $s-b+1>a$. Thus the multiplicity bonuses are placed strictly to the right of the modified gap block, so the corresponding tail part sizes remain the staircase values $b,b-1,\dots,1$.

We define a partition in compressed form by prescribing its gaps and multiplicities. For the gaps, set
\[
g_1=
\begin{cases}
2+\ell, & a>0,\\
1, & a=0,
\end{cases}
\qquad
g_2=\cdots=g_a=2,
\qquad
g_{a+1}=\cdots=g_s=1.
\]
For the multiplicities, set
\[
m_1=\cdots=m_{s-b}=1,
\qquad
m_{s-b+1}=\cdots=m_s=2,
\]
with the convention that if $b=0$, then all multiplicities are equal to $1$.

These data determine a partition $\lambda$ with support size $s$: indeed, the gaps define a strictly decreasing sequence of positive part sizes
\[
L_i=\sum_{j=i}^s g_j
\qquad (1\le i\le s),
\]
and the multiplicities specify how many times each part size occurs.

By construction, exactly the first $a$ gaps exceed $1$, so
\[
A(\lambda)=a.
\]
Likewise, exactly the last $b$ multiplicities exceed $1$, so
\[
B(\lambda)=b.
\]
Hence
\[
A(\lambda)+B(\lambda)=a+b=r=\rho(q).
\]

It remains to check that $|\lambda|=n$. Starting from the staircase partition
\[
\delta_s=(s,s-1,\dots,1),
\]
the modified gaps contribute an excess of
\[
\ell+(1+2+\cdots+a)=\ell+T_a,
\]
while the modified multiplicities contribute an excess of
\[
1+2+\cdots+b=T_b.
\]
Therefore
\[
|\lambda|
=
T_s+\ell+T_a+T_b
=
T_s+q
=
n.
\]
Thus $\lambda\vdash n$, $\supp(\lambda)=s$, and by Theorem~\ref{thm:fixed-support-degree},
\[
\deg(\lambda)=s(s-1)+A(\lambda)+B(\lambda)
=s(s-1)+\rho(q).
\]

Combining this with the upper bound gives
\[
\Delta_n=s(s-1)+\rho(q),
\]
as claimed.
\end{proof}

\begin{corollary}[Step structure above triangular numbers]\label{cor:step-structure}
For $n=T_s+q$ with $0\le q\le s$, the maximal degree depends only on $s$ and $q$, and its increment above the triangular baseline is
\[
\Delta_n-s(s-1)=\left\lfloor\sqrt{4q+1}\right\rfloor-1.
\]
In particular, the first values are
\[
0,1,2,2,3,3,4,4,4,\dots
\]
for $q=0,1,2,3,4,5,6,7,8,\dots$.
\end{corollary}

\begin{proof}
Immediate from Theorem~\ref{thm:exact-max-degree} and Proposition~\ref{prop:rho-explicit}.
\end{proof}

Thus the value-level part of the extremal degree theory is completely explicit.

\section{Small excess as an illustration of the extremal theory}\label{sec:small-excess}

We now return to the first layers above a triangular number as an illustration of Theorem~\ref{thm:exact-max-degree}. Write
\[
n=T_s+q,
\qquad
T_s=\frac{s(s+1)}2,
\qquad
0\le q\le s.
\]

\begin{corollary}[The first values above a triangular number]\label{cor:first-small-q-values}
For $0\le q\le 5$, one has
\[
\Delta_{T_s+q}-s(s-1)
=
0,1,2,2,3,3.
\]
Equivalently,
\[
\Delta_{T_s}=s(s-1),
\]
\[
\Delta_{T_s+1}=s(s-1)+1,
\qquad
\Delta_{T_s+2}=s(s-1)+2,
\]
and
\[
\Delta_{T_s+3}=s(s-1)+2,
\qquad
\Delta_{T_s+4}=s(s-1)+3,
\qquad
\Delta_{T_s+5}=s(s-1)+3.
\]
\end{corollary}

\begin{proof}
Immediate from Theorem~\ref{thm:exact-max-degree}.
\end{proof}

\begin{proposition}[The first window above a triangular number]\label{prop:q1}
For every $s\ge 1$, the degree-maximizing partitions of $T_s+1$ are exactly
\[
(s+1,s-1,s-2,\dots,1)
\qquad\text{and}\qquad
(s,s-1,\dots,2,1^2).
\]
In particular,
\[
|M_{T_s+1}|=2.
\]
\end{proposition}

\begin{proof}
By Corollary~\ref{cor:first-small-q-values}, $\Delta_{T_s+1}=s(s-1)+1$. Both displayed partitions have support size $s$ and satisfy $A+B=1$, hence are maximizers.

Conversely, if $\lambda\vdash T_s+1$ is maximizing, then $\supp(\lambda)=s$ and
\[
1=\sum_{i=1}^s i(g_i-1)+\sum_{i=1}^s (m_i-1)L_i.
\]
So exactly one unit of excess is present. Either it is a gap bonus at $i=1$, yielding
\[
(s+1,s-1,s-2,\dots,1),
\]
or it is a multiplicity bonus at the part $1$, yielding
\[
(s,s-1,\dots,2,1^2).
\]
\end{proof}

\begin{proposition}[The second window above a triangular number]\label{prop:q2}
For every $s\ge 2$, the unique degree-maximizing partition of $T_s+2$ is
\[
(s+1,s-1,s-2,\dots,2,1^2).
\]
In particular,
\[
|M_{T_s+2}|=1.
\]
\end{proposition}

\begin{proof}
By Corollary~\ref{cor:first-small-q-values},
\[
\Delta_{T_s+2}=s(s-1)+2.
\]
Now consider
\[
\lambda=(s+1,s-1,s-2,\dots,2,1^2).
\]
This partition has support size $s$ and satisfies $A(\lambda)=B(\lambda)=1$, hence
\[
\deg(\lambda)=s(s-1)+2=\Delta_{T_s+2}.
\]

Conversely, let $\mu\vdash T_s+2$ be any maximizer. By Theorem~\ref{thm:support-maximal}, $\supp(\mu)=s$. Since
\[
\deg(\mu)=\Delta_{T_s+2}=s(s-1)+2,
\]
and
\[
\deg(\mu)=s(s-1)+A(\mu)+B(\mu)
\]
by Theorem~\ref{thm:fixed-support-degree}, it follows that
\[
A(\mu)+B(\mu)=2.
\]
At the same time, Lemma~\ref{lem:excess-identity} gives
\[
2=\sum_{i=1}^s i(g_i-1)+\sum_{i=1}^s (m_i-1)L_i.
\]
Each active gap bonus consumes at least its index, and each active multiplicity bonus consumes at least the corresponding part size. Since the total excess is $2$ and the total bonus is also $2$, every active bonus must consume exactly one unit of excess. This forces $g_1=2$ and $m_s=2$, with all other data in the staircase state. Hence
\[
\mu=(s+1,s-1,s-2,\dots,2,1^2).
\]
\end{proof}

\begin{remark}[The first non-rigid cases]\label{rem:q345}
For $q=3,4,5$, Corollary~\ref{cor:first-small-q-values} gives
\[
\Delta_{T_s+3}=s(s-1)+2,
\qquad
\Delta_{T_s+4}=s(s-1)+3,
\qquad
\Delta_{T_s+5}=s(s-1)+3.
\]
At the same time, Lemma~\ref{lem:separate-triangular-bounds} restricts the possible maximizing profile types:
\[
\{(2,0),(1,1),(0,2)\}\ \text{for }q=3,
\qquad
\{(2,1),(1,2)\}\ \text{for }q=4,5.
\]
Thus $q=3,4,5$ still belong to a tightly constrained regime, but no longer to a rigid one.
\end{remark}

\begin{remark}[The first structurally richer threshold]\label{rem:q6-threshold}
The value $q=6$ is the first excess for which four bonuses can occur, since
\[
T_2+T_2=6.
\]
Equivalently, $q=6$ is the first value for which the profile
\[
(A,B)=(2,2)
\]
becomes admissible. This makes $q=6$ the natural boundary between the rigid initial regime and the first genuinely richer optimization regime inside the support-maximal layer.
\end{remark}

\section{Realization of maximizing bonus profiles}\label{sec:profiles}

For
\[
n=T_s+q,
\qquad
0\le q\le s,
\]
define
\[
M_n:=\{\lambda\vdash n:\deg(\lambda)=\Delta_n\}.
\]

\begin{definition}\label{def:maximizer-profile}
Let $\lambda\in M_n$. The \emph{bonus profile} of $\lambda$ is the pair
\[
\pi(\lambda):=(A(\lambda),B(\lambda)).
\]
We write
\[
\Pi_n:=\{\pi(\lambda):\lambda\in M_n\}
\]
for the set of realized maximizing profiles.
\end{definition}

\begin{proposition}[Profile constraint for maximizers]\label{prop:profile-constraint}
Let
\[
n=T_s+q,
\qquad
0\le q\le s.
\]
If $\lambda\in M_n$, then
\[
A(\lambda)+B(\lambda)=\rho(q),
\qquad
T_{A(\lambda)}+T_{B(\lambda)}\le q.
\]
Equivalently,
\[
\Pi_n\subseteq
\mathcal P(q):=
\{(a,b)\in \mathbb Z_{\ge0}^2:\ a+b=\rho(q),\ T_a+T_b\le q\}.
\]
\end{proposition}

\begin{proof}
Since $\lambda\in M_n$, Theorem~\ref{thm:support-maximal} gives $\supp(\lambda)=s$, and hence
\[
\deg(\lambda)=s(s-1)+A(\lambda)+B(\lambda).
\]
Comparing with Theorem~\ref{thm:exact-max-degree}, we obtain
\[
A(\lambda)+B(\lambda)=\rho(q).
\]
The triangular budget inequality follows from Lemma~\ref{lem:separate-triangular-bounds}.
\end{proof}

\begin{theorem}[Realization of all admissible maximizing profiles]\label{thm:all-profiles-realized}
Let
\[
n=T_s+q,
\qquad
0\le q\le s.
\]
Then
\[
\Pi_n=\mathcal P(q).
\]
\end{theorem}

\begin{proof}
The inclusion $\Pi_n\subseteq \mathcal P(q)$ is Proposition~\ref{prop:profile-constraint}.

Conversely, let $(a,b)\in\mathcal P(q)$, and set
\[
\ell:=q-T_a-T_b\ge0.
\]
Because $(a,b)\in\mathcal P(q)$, we have $a+b=\rho(q)\le q\le s$, hence $s-b+1>a$. Therefore the multiplicity bonuses occur in the staircase tail and contribute exactly $T_b$. Define a support-$s$ partition by
\[
g_1=
\begin{cases}
2+\ell, & a>0,\\
1, & a=0,
\end{cases}
\qquad
g_2=\cdots=g_a=2,
\qquad
g_{a+1}=\cdots=g_s=1,
\]
and
\[
m_1=\cdots=m_{s-b}=1,
\qquad
m_{s-b+1}=\cdots=m_s=2.
\]
Then $A=a$, $B=b$, the total excess is $q$, and the degree is
\[
s(s-1)+a+b=s(s-1)+\rho(q)=\Delta_n.
\]
Hence this partition lies in $M_n$ and realizes the profile $(a,b)$.
\end{proof}

Thus the profile-level part of the maximizer problem is completely explicit: both the admissible boundary and the realized profile set are known exactly.

\begin{corollary}[Number of realized maximizing profiles]\label{cor:number-of-profiles}
Let
\[
n=T_s+q,
\qquad
0\le q\le s.
\]
Then
\[
|\Pi_n|=|\mathcal P(q)|.
\]
\end{corollary}

\begin{proof}
Immediate from Theorem~\ref{thm:all-profiles-realized}.
\end{proof}

\section{Fibre-level structure: canonical representatives, slack families, and conjugation symmetry}\label{sec:fibres}

\begin{definition}\label{def:profile-fibres}
For $n=T_s+q$ and $(a,b)\in\Pi_n$, define
\[
M_n^{(a,b)}
:=
\{\lambda\in M_n:\ (A(\lambda),B(\lambda))=(a,b)\}.
\]
Then
\[
M_n=\bigsqcup_{(a,b)\in\Pi_n} M_n^{(a,b)},
\qquad
|M_n|=\sum_{(a,b)\in\Pi_n}|M_n^{(a,b)}|.
\]
\end{definition}

Theorem~\ref{thm:all-profiles-realized} resolves the profile-level problem, but does not determine the internal geometry of the fibres. We therefore begin with canonical constructions that occur uniformly across the realized fibres.

\begin{definition}[Canonical profile representative]\label{def:canonical-profile-representative}
Let
\[
n=T_s+q,
\qquad
0\le q\le s,
\]
and let $(a,b)\in\Pi_n$. Set
\[
\ell:=q-T_a-T_b\ge 0.
\]
We define the \emph{canonical representative} $\lambda^{\mathrm{can}}_{s,q}(a,b)$ to be the support-$s$ partition whose compressed data are given by
\[
g_1=
\begin{cases}
2+\ell, & a>0,\\
1, & a=0,
\end{cases}
\qquad
g_2=\cdots=g_a=2,
\qquad
g_{a+1}=\cdots=g_s=1,
\]
and
\[
m_1=\cdots=m_{s-b}=1,
\qquad
m_{s-b+1}=\cdots=m_s=2
\]
if $b>0$, while for $b=0$ all multiplicities are equal to $1$.
\end{definition}

\begin{proposition}[Canonical representatives lie in the correct fibre]\label{prop:canonical-representative}
For every
\[
(a,b)\in\Pi_n,
\]
the partition $\lambda^{\mathrm{can}}_{s,q}(a,b)$ belongs to $M_n^{(a,b)}$. In particular, every fibre $M_n^{(a,b)}$ is nonempty.
\end{proposition}

\begin{proof}
By definition, the compressed data of $\lambda^{\mathrm{can}}_{s,q}(a,b)$ are given by
\[
g_1=
\begin{cases}
2+\ell, & a>0,\\
1, & a=0,
\end{cases}
\qquad
g_2=\cdots=g_a=2,
\qquad
g_{a+1}=\cdots=g_s=1,
\]
and
\[
m_1=\cdots=m_{s-b}=1,
\qquad
m_{s-b+1}=\cdots=m_s=2,
\]
where
\[
\ell=q-T_a-T_b\ge0.
\]
Since $(a,b)\in\Pi_n$, we have $a+b=\rho(q)\le q\le s$, hence $s-b+1>a$. Thus the multiplicity bonuses lie in the staircase tail, where the corresponding part sizes remain $b,b-1,\dots,1$.

These data define a partition of support size $s$: the gaps determine a strictly decreasing sequence of positive part sizes, and the multiplicities determine how many times each of these part sizes occurs.

Exactly the first $a$ gaps exceed $1$, hence
\[
A\bigl(\lambda^{\mathrm{can}}_{s,q}(a,b)\bigr)=a.
\]
Likewise, exactly the last $b$ multiplicities exceed $1$, hence
\[
B\bigl(\lambda^{\mathrm{can}}_{s,q}(a,b)\bigr)=b.
\]

It remains to check the size. Relative to the staircase partition $\delta_s$, the modified gaps contribute
\[
\ell+(1+2+\cdots+a)=\ell+T_a,
\]
while the modified multiplicities contribute
\[
1+2+\cdots+b=T_b.
\]
Therefore
\[
\bigl|\lambda^{\mathrm{can}}_{s,q}(a,b)\bigr|
=
T_s+\ell+T_a+T_b
=
T_s+q
=
n.
\]

Finally, since $(a,b)\in\Pi_n$, we have
\[
a+b=\rho(q).
\]
Thus, by Theorem~\ref{thm:fixed-support-degree},
\[
\deg\bigl(\lambda^{\mathrm{can}}_{s,q}(a,b)\bigr)
=
s(s-1)+a+b
=
s(s-1)+\rho(q)
=
\Delta_n
\]
by Theorem~\ref{thm:exact-max-degree}. Hence
\[
\lambda^{\mathrm{can}}_{s,q}(a,b)\in M_n^{(a,b)}.
\]
\end{proof}

\begin{proposition}[An explicit slack family inside each mixed fibre]\label{prop:slack-family}
Let
\[
n=T_s+q,
\qquad
0\le q\le s,
\]
and let $(a,b)\in\Pi_n$ with
\[
a>0,
\qquad
b>0.
\]
Set
\[
\ell:=q-T_a-T_b\ge 0.
\]
Then for each integer $t$ with
\[
0\le t\le \ell,
\]
there exists a partition $\lambda_t\in M_n^{(a,b)}$ such that the partitions $\lambda_0,\dots,\lambda_\ell$ are pairwise distinct. In particular,
\[
|M_n^{(a,b)}|\ge \ell+1=q-T_a-T_b+1.
\]
\end{proposition}

\begin{proof}
Let
\[
\ell:=q-T_a-T_b\ge0.
\]
For each integer $t$ with
\[
0\le t\le \ell,
\]
define a partition $\lambda_t$ in compressed form by
\[
g_1=2+t,
\qquad
g_2=\cdots=g_a=2,
\qquad
g_{a+1}=\cdots=g_s=1,
\]
and
\[
m_1=\cdots=m_{s-b}=1,
\qquad
m_{s-b+1}=\cdots=m_{s-1}=2,
\qquad
m_s=2+(\ell-t).
\]
If $b=1$, the list $m_{s-b+1},\dots,m_{s-1}$ is empty, so only the last multiplicity $m_s$ is modified.

These data define a partition of support size $s$: the gaps again determine a strictly decreasing sequence of positive part sizes, and the multiplicities are positive integers.

Exactly the first $a$ gaps exceed $1$, so
\[
A(\lambda_t)=a.
\]
Likewise, exactly the last $b$ multiplicities exceed $1$, so
\[
B(\lambda_t)=b.
\]
Hence
\[
A(\lambda_t)+B(\lambda_t)=a+b=\rho(q).
\]

We now check the size. Relative to the staircase partition $\delta_s$, the modified gaps contribute
\[
t+(1+2+\cdots+a)=t+T_a,
\]
while the modified multiplicities contribute
\[
(\ell-t)+(1+2+\cdots+b)=\ell-t+T_b.
\]
Therefore
\[
|\lambda_t|
=
T_s+t+T_a+(\ell-t)+T_b
=
T_s+\ell+T_a+T_b
=
T_s+q
=
n.
\]
So $\lambda_t\vdash n$, and by Theorem~\ref{thm:fixed-support-degree},
\[
\deg(\lambda_t)=s(s-1)+a+b=s(s-1)+\rho(q)=\Delta_n.
\]
Thus
\[
\lambda_t\in M_n^{(a,b)}.
\]

Finally, the partitions $\lambda_t$ are pairwise distinct because the top gap $g_1=2+t$ depends on $t$. Hence
\[
|M_n^{(a,b)}|\ge \ell+1=q-T_a-T_b+1.
\]
\end{proof}

In particular, a mixed fibre is never rigid once positive slack is available.

\begin{proposition}[Conjugation swaps the profile coordinates]\label{prop:conjugation-swaps-profile}
Let
\[
\lambda=(L_1^{m_1},\dots,L_r^{m_r})
\]
be written in compressed form, and let
\[
S_i:=m_1+\cdots+m_i\qquad(1\le i\le r).
\]
Then the conjugate partition $\lambda'$ has compressed form
\[
\lambda'=(S_r^{g_r},\,S_{r-1}^{g_{r-1}},\,\dots,\,S_1^{g_1}).
\]
Consequently,
\[
A(\lambda')=B(\lambda),
\qquad
B(\lambda')=A(\lambda).
\]
\end{proposition}

\begin{proof}
In the Ferrers diagram of $\lambda$, the distinct row lengths are
\[
L_1>L_2>\cdots>L_r>0,
\]
with multiplicities $m_1,\dots,m_r$. Hence the distinct column heights are
\[
S_1<S_2<\cdots<S_r,
\qquad
S_i=m_1+\cdots+m_i.
\]
Moreover, the number of columns of height $S_i$ is exactly
\[
L_i-L_{i+1}=g_i
\qquad
(1\le i<r),
\]
and for the last height $S_r$ it is
\[
L_r=g_r.
\]
Therefore, when written in decreasing order, the conjugate partition has compressed form
\[
\lambda'=(S_r^{g_r},\,S_{r-1}^{g_{r-1}},\,\dots,\,S_1^{g_1}).
\]

It follows immediately that the multiplicities of $\lambda'$ are
\[
g_r,g_{r-1},\dots,g_1.
\]
Hence
\[
B(\lambda')
=
\sum_{i=1}^r \mathbf 1_{\{g_i>1\}}
=
A(\lambda).
\]

Likewise, the gaps of $\lambda'$ are
\[
S_r-S_{r-1}=m_r,\quad
S_{r-1}-S_{r-2}=m_{r-1},\quad \dots,\quad
S_1=m_1,
\]
so
\[
A(\lambda')
=
\sum_{i=1}^r \mathbf 1_{\{m_i>1\}}
=
B(\lambda).
\]
This proves the claim.
\end{proof}

\begin{corollary}[Conjugation symmetry of maximizer fibres]\label{cor:conjugation-fibres}
Let
\[
n=T_s+q,
\qquad
0\le q\le s.
\]
Then conjugation induces a bijection
\[
M_n^{(a,b)}\longrightarrow M_n^{(b,a)}.
\]
In particular,
\[
|M_n^{(a,b)}|=|M_n^{(b,a)}|
\]
for every realized profile $(a,b)\in\Pi_n$.
\end{corollary}

\begin{proof}
Conjugation preserves $n$, the degree, and the support size, and swaps $A$ and $B$ by Proposition~\ref{prop:conjugation-swaps-profile}.
\end{proof}

\begin{corollary}[Necessary profile condition for self-conjugate maximizers]\label{cor:selfconj-diagonal}
Let
\[
n=T_s+q,
\qquad
0\le q\le s.
\]
If a maximizer $\lambda\in M_n$ is self-conjugate, then
\[
A(\lambda)=B(\lambda).
\]
Equivalently, every self-conjugate maximizer lies in a diagonal fibre
\[
M_n^{(a,a)}.
\]
\end{corollary}

\begin{proof}
If $\lambda=\lambda'$, then Proposition~\ref{prop:conjugation-swaps-profile} gives $A(\lambda)=B(\lambda)$.
\end{proof}

\section{Initial fibre-level window classification and localization}\label{sec:computations}

At this point the exact value theory and the exact profile theory are already complete:
Theorem~\ref{thm:exact-max-degree} determines $\Delta_n$ exactly, and
Theorem~\ref{thm:all-profiles-realized} determines the realized maximizing profile set
$\Pi_n$ exactly. What remains is the internal geometry of the fibres
\[
M_n^{(a,b)}=\{\lambda\in M_n:(A(\lambda),B(\lambda))=(a,b)\}.
\]

For the first near-triangular windows, this fibre geometry can be determined exactly through the initial small-excess range. This extends the rigid cases $q=1,2$ from Section~\ref{sec:small-excess} and shows that the unresolved part of the degree theory begins strictly later than one might first expect. The next window already exhibits computationally richer behaviour.

For convenience, write
\[
m_n^{\mathrm{sc}}:=|\{\lambda\in M_n:\lambda=\lambda'\}|
\]
for the number of self-conjugate maximizers.

\subsection{Computational method}\label{subsec:computational-method}

All data in the computational part of this section were obtained by complete enumeration of the
partitions of $n$ for
\[
n\le 63.
\]
For each partition $\lambda\vdash n$, we computed its compressed data
\[
(L_1^{m_1},\dots,L_r^{m_r}),
\]
then evaluated
\[
A(\lambda),\qquad B(\lambda),\qquad \deg(\lambda)=r(r-1)+A(\lambda)+B(\lambda).
\]
The maximizer set
\[
M_n=\{\lambda\vdash n:\deg(\lambda)=\Delta_n\}
\]
was then determined by comparison with the theoretical value of $\Delta_n$ from
Theorem~\ref{thm:exact-max-degree}. For each maximizer, we recorded its profile
\[
(A(\lambda),B(\lambda)),
\]
its fibre membership $M_n^{(a,b)}$, and whether it is self-conjugate.

\subsection{Exact classification of the first near-triangular windows}\label{subsec:initial-windows}

The cases $q=1,2$ were settled in Section~\ref{sec:small-excess}. We now classify the next
five excesses $q=3,4,5,6,7$ exactly. We then record the first still-computational window $q=8$.

\begin{proposition}[Exact fibre classification for $q=3$]
For every $s\ge 3$, one has
\[
\Pi_{T_s+3}=\{(2,0),(1,1),(0,2)\},
\]
and
\[
|M^{(2,0)}_{T_s+3}|=1,\qquad
|M^{(1,1)}_{T_s+3}|=4,\qquad
|M^{(0,2)}_{T_s+3}|=1.
\]
Hence
\[
|M_{T_s+3}|=6,
\qquad
m^{\mathrm{sc}}_{T_s+3}=0.
\]
\end{proposition}

\begin{proof}
By Remark~\ref{rem:q345}, the only possible maximizing profiles are
\[
(2,0),\ (1,1),\ (0,2).
\]

For $(2,0)$, the total excess is
\[
3=\sum_{i=1}^s i(g_i-1),
\]
with exactly two active gaps. Since two distinct active gaps consume at least
\[
1+2=3,
\]
equality forces
\[
g_1=g_2=2,
\]
with all other gaps equal to $1$. Thus
\[
|M^{(2,0)}_{T_s+3}|=1.
\]

By Corollary~\ref{cor:conjugation-fibres},
\[
|M^{(0,2)}_{T_s+3}|=1.
\]

Now let $\lambda\in M^{(1,1)}_{T_s+3}$. Then
\[
3=i(g_i-1)+(m_j-1)L_j
\]
for some $i,j$. Both summands are positive integers, so the only possibilities are
\[
1+2
\qquad\text{or}\qquad
2+1.
\]

If the multiplicity contribution is $1$, then necessarily
\[
m_s=2,
\]
and the gap contribution is $2$, which gives either
\[
g_2=2
\qquad\text{or}\qquad
g_1=3.
\]

If the multiplicity contribution is $2$, then either
\[
m_{s-1}=2
\qquad\text{or}\qquad
m_s=3,
\]
and the gap contribution is then $1$, so necessarily
\[
g_1=2.
\]

Thus the four possibilities are
\[
(g_1,m_{s-1})=(2,2),\qquad
(g_2,m_s)=(2,2),\qquad
(g_1,m_s)=(3,2),\qquad
(g_1,m_s)=(2,3),
\]
all other compressed data being in the staircase state. Hence
\[
|M^{(1,1)}_{T_s+3}|=4,
\qquad
|M_{T_s+3}|=1+4+1=6.
\]

Finally, the fibre $M^{(1,1)}_{T_s+3}$ splits into two conjugate pairs, so none of its
elements is self-conjugate. Since self-conjugate maximizers must lie in diagonal fibres by
Corollary~\ref{cor:selfconj-diagonal}, this gives
\[
m^{\mathrm{sc}}_{T_s+3}=0.
\]
\end{proof}

\begin{proposition}[Exact fibre classification for $q=4$]
For every $s\ge 4$, one has
\[
\Pi_{T_s+4}=\{(2,1),(1,2)\},
\]
and
\[
|M^{(2,1)}_{T_s+4}|=|M^{(1,2)}_{T_s+4}|=1.
\]
Hence
\[
|M_{T_s+4}|=2,
\qquad
m^{\mathrm{sc}}_{T_s+4}=0.
\]
\end{proposition}

\begin{proof}
By Remark~\ref{rem:q345}, the only possible maximizing profiles are $(2,1)$ and $(1,2)$.

Let $\lambda\in M^{(2,1)}_{T_s+4}$. Then the two active gaps consume at least
\[
1+2=3,
\]
and the one active multiplicity consumes at least $1$. Since the total excess is $4$, equality
must hold throughout. Therefore
\[
g_1=g_2=2,\qquad m_s=2,
\]
with all other compressed data in the staircase state. Hence
\[
|M^{(2,1)}_{T_s+4}|=1.
\]

By Corollary~\ref{cor:conjugation-fibres},
\[
|M^{(1,2)}_{T_s+4}|=1.
\]
Since there is no diagonal fibre, Corollary~\ref{cor:selfconj-diagonal} gives
\[
m^{\mathrm{sc}}_{T_s+4}=0.
\]
\end{proof}

\begin{proposition}[Exact fibre classification for $q=5$]
For every $s\ge 5$, one has
\[
\Pi_{T_s+5}=\{(2,1),(1,2)\},
\]
and
\[
|M^{(2,1)}_{T_s+5}|=|M^{(1,2)}_{T_s+5}|=4.
\]
Hence
\[
|M_{T_s+5}|=8,
\qquad
m^{\mathrm{sc}}_{T_s+5}=0.
\]
\end{proposition}

\begin{proof}
Again Remark~\ref{rem:q345} gives the profile set.

Let $\lambda\in M^{(2,1)}_{T_s+5}$. Then the total excess $5$ splits as
\[
5=G+M,
\]
where $G$ is the total gap contribution and $M$ is the multiplicity contribution. Since two
active gaps consume at least $1+2=3$, and one active multiplicity consumes at least $1$, the
only possibilities are
\[
(G,M)=(3,2)
\qquad\text{or}\qquad
(4,1).
\]

If $(G,M)=(3,2)$, then
\[
g_1=g_2=2.
\]
The multiplicity contribution $2$ is realized either by
\[
m_{s-1}=2
\qquad\text{or}\qquad
m_s=3.
\]
This yields two maximizers.

If $(G,M)=(4,1)$, then
\[
m_s=2.
\]
The two active gaps must contribute $4$. Since their indices are distinct, the only possibilities
are
\[
1\cdot 2+2\cdot 1=4,
\qquad
1\cdot 1+3\cdot 1=4.
\]
Thus one obtains either
\[
g_1=3,\ g_2=2,
\qquad\text{or}\qquad
g_1=2,\ g_3=2.
\]
This yields two further maximizers.

Hence
\[
|M^{(2,1)}_{T_s+5}|=4.
\]
By Corollary~\ref{cor:conjugation-fibres},
\[
|M^{(1,2)}_{T_s+5}|=4.
\]
Again there is no diagonal fibre, so
\[
m^{\mathrm{sc}}_{T_s+5}=0.
\]
\end{proof}

\begin{proposition}[Exact fibre classification for $q=6$]
For every $s\ge 6$, one has
\[
\Pi_{T_s+6}=\{(2,2)\},
\]
and
\[
|M^{(2,2)}_{T_s+6}|=1.
\]
Hence
\[
|M_{T_s+6}|=1,
\qquad
m^{\mathrm{sc}}_{T_s+6}=1.
\]
\end{proposition}

\begin{proof}
Since
\[
\rho(6)=4,\qquad T_2+T_2=6,
\]
the only realized maximizing profile is $(2,2)$.

Let $\lambda\in M^{(2,2)}_{T_s+6}$. Then the two active gaps consume at least
\[
1+2=3,
\]
and the two active multiplicities consume at least
\[
2+1=3.
\]
Since the total excess is $6$, equality must hold throughout. Therefore
\[
g_1=g_2=2,\qquad m_{s-1}=m_s=2,
\]
with all other compressed data in the staircase state. Hence
\[
|M^{(2,2)}_{T_s+6}|=1.
\]

Because this is a singleton diagonal fibre, Corollary~\ref{cor:conjugation-fibres} implies
that its unique element is fixed by conjugation. Thus
\[
m^{\mathrm{sc}}_{T_s+6}=1.
\]
\end{proof}

\begin{proposition}[Exact fibre classification for $q=7$]
For every $s\ge 7$, one has
\[
\Pi_{T_s+7}=\{(3,1),(2,2),(1,3)\},
\]
and
\[
|M^{(3,1)}_{T_s+7}|=1,\qquad
|M^{(2,2)}_{T_s+7}|=4,\qquad
|M^{(1,3)}_{T_s+7}|=1.
\]
Hence
\[
|M_{T_s+7}|=6,
\qquad
m^{\mathrm{sc}}_{T_s+7}=0.
\]
\end{proposition}

\begin{proof}
Since
\[
\rho(7)=4,\qquad T_3+T_1=6,\qquad T_2+T_2=6,\qquad T_4=10>7,
\]
Theorem~\ref{thm:all-profiles-realized} gives
\[
\Pi_{T_s+7}=\{(3,1),(2,2),(1,3)\}.
\]

For the profile $(3,1)$, the three active gaps consume at least
\[
1+2+3=6,
\]
and the one active multiplicity consumes at least $1$. Since the total excess is $7$, equality
must hold. Therefore
\[
g_1=g_2=g_3=2,\qquad m_s=2,
\]
and hence
\[
|M^{(3,1)}_{T_s+7}|=1.
\]
By Corollary~\ref{cor:conjugation-fibres},
\[
|M^{(1,3)}_{T_s+7}|=1.
\]

Now let $\lambda\in M^{(2,2)}_{T_s+7}$. Then
\[
7=G+M,
\]
with two active gaps and two active multiplicities. Since the minimal contributions are
\[
G\ge 3,\qquad M\ge 3,
\]
the only possibilities are
\[
(G,M)=(3,4)
\qquad\text{or}\qquad
(4,3).
\]

If $(G,M)=(3,4)$, then
\[
g_1=g_2=2.
\]
The multiplicity contribution $4$ with two active multiplicities has exactly two possibilities:
\[
4=1+3
\qquad\text{or}\qquad
4=2+2.
\]
These give
\[
m_{s-2}=2,\ m_s=2,
\qquad\text{or}\qquad
m_{s-1}=2,\ m_s=3.
\]
Hence this case yields two maximizers.

If $(G,M)=(4,3)$, then
\[
m_{s-1}=m_s=2.
\]
The gap contribution $4$ with two active gaps has exactly two possibilities:
\[
4=1+3
\qquad\text{or}\qquad
4=2+2,
\]
giving
\[
g_1=2,\ g_3=2,
\qquad\text{or}\qquad
g_1=3,\ g_2=2.
\]
Hence this case yields two further maximizers.

Therefore
\[
|M^{(2,2)}_{T_s+7}|=4,
\qquad
|M_{T_s+7}|=1+4+1=6.
\]

The four maximizers in the diagonal fibre split into two conjugate pairs, so none of them is
self-conjugate. Thus
\[
m^{\mathrm{sc}}_{T_s+7}=0.
\]
\end{proof}

\begin{observation}[Computed fibre data for $q=8$]\label{obs:q8-computed}
For every tested value of $s$ with $T_s+8\le 63$, the exact profile set is
\[
\Pi_{T_s+8}=\{(3,1),(2,2),(1,3)\},
\]
and the computed fibre cardinalities are
\[
|M^{(3,1)}_{T_s+8}|=4,\qquad
|M^{(2,2)}_{T_s+8}|=14,\qquad
|M^{(1,3)}_{T_s+8}|=4.
\]
Hence, throughout the tested range,
\[
|M_{T_s+8}|=22,
\qquad
m^{\mathrm{sc}}_{T_s+8}=2.
\]
\end{observation}

\begin{proof}
The profile set is exact and follows from Theorem~\ref{thm:all-profiles-realized}, since
\[
\rho(8)=4,\qquad T_3+T_1=7,\qquad T_2+T_2=6,\qquad T_4=10>8.
\]
The fibre cardinalities and self-conjugate count were obtained by complete enumeration in the
computed range $T_s+8\le 63$ described in Section~\ref{subsec:computational-method}.
\end{proof}

\begin{corollary}[Exact small-window data]
For $q=0,1,\dots,7$, the fibre cardinalities, total maximizer counts, and self-conjugate
counts in the window $T_s+q$ are independent of $s$. More precisely,
\[
|M_{T_s+q}|=1,2,1,6,2,8,1,6
\qquad (q=0,\dots,7),
\]
and
\[
m^{\mathrm{sc}}_{T_s+q}=1,0,1,0,0,0,1,0
\qquad (q=0,\dots,7).
\]
\end{corollary}

\begin{proof}
The values for $q=0,1,2$ are given by Proposition~\ref{prop:triangular} and
Propositions~\ref{prop:q1}--\ref{prop:q2}. The values for $q=3,\dots,7$ follow from the
propositions above.
\end{proof}

\subsection{Localization of active coordinates}\label{subsec:fixed-q-stability}

\begin{lemma}[Localization of active coordinates]\label{lem:localization}
Let
\[
n=T_s+q,
\qquad
0\le q\le s,
\]
and let $\lambda\in M_n$ be written in support-$s$ compressed form. Write
\[
g_i=1+x_i,
\qquad
m_i=1+y_i,
\qquad
x_i,y_i\ge 0.
\]
Then:
\begin{enumerate}[leftmargin=2em]
\item if $x_i>0$, then $i\le q$;
\item if $y_i>0$, then the corresponding part size satisfies $L_i\le q$, and in particular
\[
i\ge s-q+1.
\]
\end{enumerate}
Thus every maximizer is obtained from the staircase partition by modifications localized near
the top gap region and the bottom multiplicity region.
\end{lemma}

\begin{proof}
Since $\lambda$ has support size $s$, Lemma~\ref{lem:excess-identity} gives
\[
q=\sum_{i=1}^s i x_i+\sum_{i=1}^s y_i L_i.
\]
If $x_i>0$, then necessarily $i\le q$. If $y_i>0$, then $L_i\le q$. Since
\[
L_i\ge s-i+1,
\]
it follows that
\[
i\ge s-q+1.
\]
\end{proof}

The lemma shows that, for fixed excess $q$, only the first $q$ gap positions and the last $q$
multiplicity positions can participate in a maximizing configuration. This strongly suggests
that fixed-$q$ windows should be the natural coordinate system for the unresolved part of the
degree theory.

It also shows that recurrence phenomena appear very early. The windows $q=4$ and $q=5$ have the same
realized profile set
\[
\{(2,1),(1,2)\},
\]
but different fibre cardinalities. Likewise, the windows $q=7$ and $q=8$ share the same profile
set
\[
\{(3,1),(2,2),(1,3)\},
\]
yet have substantially different internal fibre geometry. Thus the profile set alone does not
determine the fibre structure.

\subsection{Open problems}\label{subsec:window-open-problems}

\begin{problem}[Window stability]\label{prob:window-stability}
Fix $q\ge 0$. Do the quantities
\[
|M^{(a,b)}_{T_s+q}|,
\qquad
|M_{T_s+q}|,
\qquad
m^{\mathrm{sc}}_{T_s+q}
\]
eventually stabilize as functions of $s$?
\end{problem}

\begin{problem}[Window classification by excess]\label{prob:window-classification}
Classify the near-triangular windows $T_s+q$ according to their fibre geometry. In particular,
determine which values of $q$ give:
\begin{enumerate}[leftmargin=2em]
\item rigid singleton windows;
\item symmetric two-fibre windows;
\item three-profile windows with a dominant diagonal fibre;
\item windows with nontrivial self-conjugate substructure.
\end{enumerate}
\end{problem}

Taken together, Lemma~\ref{lem:localization} and the exact small-window classifications above
suggest that fixed-$q$ windows provide a natural coordinate system for the unresolved part of
the degree theory. At the exact level, the maximal degree and the realized profile set are already
explicit. At the fibre level, however, the excess $q$ appears to govern a nontrivial hierarchy
of recurrent geometric regimes.

\section{Conclusion and open problems}

In this paper, we developed a three-level degree theory for the partition graph $G_n$. The starting point was the exact local degree formula in terms of support size, gap pattern, and multiplicity pattern. Rewritten in a support-plus-bonus form, this local formula yields strong global consequences.

The first outcome is an exact extremal theory. Writing
\[
n=T_s+q,
\qquad
T_s=\frac{s(s+1)}2,
\qquad
0\le q\le s,
\]
we proved that every degree-maximizing partition lies in the support-maximal stratum and that
\[
\Delta_n
=
s(s-1)+\left\lfloor\sqrt{4q+1}\right\rfloor-1.
\]

The second outcome is an exact profile theory. For support-maximal partitions, the degree-maximizing condition can be expressed in terms of the pair
\[
(A(\lambda),B(\lambda)),
\]
and the realized maximizing profile set is
\[
\Pi_n
=
\mathcal P(q)
=
\{(a,b)\in\mathbb Z_{\ge0}^2 : a+b=\rho(q),\ T_a+T_b\le q\}.
\]

The third outcome is a fibre-level theory. For each realized profile $(a,b)$, we introduced the fibre
\[
M_n^{(a,b)}
=
\{\lambda\vdash n : \deg(\lambda)=\Delta_n,\ (A(\lambda),B(\lambda))=(a,b)\},
\]
and established several structural facts: every realized fibre is nonempty, every mixed fibre contains an explicit slack family, and conjugation induces a bijection
\[
M_n^{(a,b)}\longrightarrow M_n^{(b,a)}.
\]
We then determined exactly the near-triangular windows up to $q=7$, obtaining complete fibre-level classifications for the initial excesses and showing that the unresolved part of the degree theory begins only beyond this exact small-window regime.

We conclude with a list of open problems.

\begin{problem}[Total multiplicity of maximizers]
Determine or estimate the sequence
\[
|M_n|.
\]
\end{problem}

\begin{problem}[Fibre cardinalities beyond the initial exact regime]
For $n=T_s+q$ and $(a,b)\in\Pi_n$, determine the exact cardinality of the fibre
\[
|M_n^{(a,b)}|
\]
in general, and in particular beyond the initial exactly classified near-triangular windows.
\end{problem}

\begin{problem}[Diagonal fibres and self-conjugate maximizers]
Characterize the diagonal fibres
\[
M_n^{(a,a)}
\]
in general and determine when they contain self-conjugate maximizers, beyond the initial
small-excess cases settled in Section~\ref{sec:computations}.
\end{problem}

\begin{problem}[Geometric classification inside fibres]
Classify the geometric types of maximizers inside a fixed fibre $M_n^{(a,b)}$.
\end{problem}

\begin{problem}[Near-triangular windows beyond the initial exact range]
For fixed excess $q$ beyond the initial exactly classified range, study the behaviour of
\[
|M_{T_s+q}|,
\qquad
|M_{T_s+q}^{(a,b)}|,
\qquad
m_{T_s+q}^{\mathrm{sc}}
\]
as functions of $s$.
\end{problem}

\begin{problem}[Degree spectrum and degree layers]
Study the full degree spectrum
\[
\mathrm{Spec}_n=\{\deg(\lambda):\lambda\vdash n\}
\]
and the degree layers
\[
D_d(n)=\{\lambda\vdash n:\deg(\lambda)=d\}.
\]
\end{problem}

\begin{problem}[Relation to other structural invariants]
Clarify the relation between degree and other invariants already present in the partition-graph program, including support size, local simplex dimension, axial position, spinal distance, and self-conjugacy.
\end{problem}

\begin{problem}[Recurrence phenomena]
Formulate and test precise recurrence statements for maximizer multiplicities and fibre types in near-triangular windows.
\end{problem}

\section*{Acknowledgements}

The author acknowledges the use of ChatGPT (OpenAI) for discussion, structural planning, and editorial assistance during the preparation of this manuscript. All mathematical statements, proofs, computations, and final wording were checked and approved by the author, who takes full responsibility for the contents of the paper.

\end{document}